\documentclass{amsart}

\usepackage{amsmath,amssymb,amsthm}
\usepackage{hyperref}
\usepackage{xcolor}
\usepackage{bbm}

\newtheorem{thmintr}{Theorem}

\newtheorem{prop}{Proposition}[subsection]

\newtheorem{lem}[prop]{Lemma}
\newtheorem{cor}[prop]{Corollary}

\theoremstyle{definition}

\newtheorem{rem}[prop]{Remark}
\newtheorem{exwith}[prop]{Example}

\newtheorem*{ack}{Acknowledgement}

\newcommand{\bfa}{\mathbf a}

\newcommand{\bfb}{\mathbf b}

\newcommand{\C}{\mathbb C}

\newcommand{\rmd}{\mathrm d}

\newcommand{\rme}{\mathrm e}

\newcommand{\rmi}{\mathrm i}

\newcommand{\N}{\mathbb N}

\newcommand{\Q}{\mathbb Q}

\newcommand{\bfz}{\mathbf z}

\DeclareMathOperator{\lcm}{\mathrm{lcm}}

\DeclareMathOperator{\RS}{\mathrm{RS}}

\newcommand{\chim}{\chi_{\mathrm{m}}}

\begin{document}

\author{Florian Buck}
\address{Fakult\"at f\"ur Mathematik, Ruhr-Universit\"at Bochum,
Universit\"atsstra{\ss}e 150, D-44801 Bochum, Germany}
\email{Florian.Buck@rub.de, Kai.Zehmisch@rub.de}
\author{Kai Zehmisch}

\title[Connected sums of Brieskorn contact $5$-spheres]
{Connected sums of Brieskorn contact $5$-spheres}

\date{07.07.2025}

\begin{abstract}
  In dimension $5$,
  the contact connected sum of Brieskorn contact spheres
  is, in general, not a Brieskorn contact sphere.
\end{abstract}

\subjclass[2020]{53D42; 53D40, 37C27, 37J55, 57R17.}
\thanks{This work is part of a project in the SFB/TRR 191
{\it Symplectic Structures in Geometry, Algebra and Dynamics},
funded by the DFG}

\maketitle



\section{Introduction\label{sec:intro}}

{\bf Brieskorn manifolds} $\Sigma(\bfa)$,
initially studied by Brieskorn \cite{brk66},
Hirzebruch--Mayer \cite{hm68}
and Milnor \cite{mil68}
in the context of exotic differentiable structures on spheres,
are given as link $\{p_{\bfa}=0\}\cap S^{2n+1}$
of singularity of complex polynomials
\[
p_{\bfa}(z_0,\ldots,z_n)=
z_0^{a_0}+\cdots+z_n^{a_n}
\,,
\qquad
\bfa=(a_0,\ldots,a_n)
\,.
\]
Many of them, the so-called {\bf Brieskorn spheres},
are homotopy equivalent to the unit sphere
$S^{2n+1}\subset\C^{n+1}$,
cf.\ Section \ref{subsec:brieskorncriterion} below.
Each Brieskorn manifold $\Sigma(\bfa)$
carries a Stein fillable contact structure $\xi_{\bfa}$,
which we will call a {\bf Brieskorn contact structure}.
The contact structure $\xi_{\bfa}$
is defined by intersecting the standard contact structure
$TS^{2n+1}\cap\rmi TS^{2n+1}$ of $S^{2n+1}$
with $T\Sigma(\bfa)$, see \cite[Section 7.1]{gei08}.
In case $\Sigma(\bfa)$ is diffeomorphic
to the standard sphere $S^{2n-1}$,
exoticity questions of the contact structures $\xi_{\bfa}$
become relevant contact-topologically.
Here,
a contact structure on $S^{2n-1}\subset\C^n$
not contactomorphic to $TS^{2n-1}\cap\rmi TS^{2n-1}$
is called {\bf exotic}.

Bennequin \cite{benn83} found the first examples
of exotic contact structures on $S^3$.
Eliashberg \cite{eli89} observed that these examples are
overtwisted contact structures,
for which an $h$-principle holds and which are not fillable,
see \cite{elia90,gro85} and cf.\ \cite[Corollary 3.8]{gz13}.
Extensions of overtwistedness and non-fillability phenomena
to higher dimensions also hold true,
see \cite{bem15} and \cite{nie06, mnw13,schsz25}, resp.

Eliashberg \cite{elia91}, Geiges \cite{gei97} and Ding--Geiges \cite{dg04}
gave examples of exotic contact structures on $S^{2n-1}$ for $n\geq3$.
Exoticity here was established
via the Eliashberg--Floer--McDuff theorem
\cite[Theorem 1.5]{mcd91}
about smooth uniqueness of symplectically aspherical fillings
of the standard contact sphere;
such a filling must be diffeomorphic to a ball,
see also \cite{bgz19,zhou23}.
Further aspects of exoticity phenomena in terms of fillability,
Reeb dynamics and tightness were investigated in
\cite{mcl11,kvk16,am19,laz20,bgmz24}.

Exotic Brieskorn contact structures on standard spheres do exist.
Ustilovsky \cite{ust99} showed,
based on contact homology as worked out e.g.\ by Pardon \cite{par19},
that on $S^{4m+1}$, $m\in\N$,
there exists infinitely many pairwise non-contactomorphic
Brieskorn contact structures
in every homotopy class of almost contact structures.
Fauck \cite{fau15} obtained the same result by alternatively using
Rabinowitz--Floer homology;
Gutt \cite{gutt17} (and in parts Uebele \cite{ueb16}) via
positive equivariant symplectic homology.

Kwon--van Koert \cite{kvk16}
extended these results, which in particular in dimension $5$
are of interest by the following curiosities:
In dimension $5$, no exotic smooth spheres exist
according to the generalised Poincar\'e conjecture
(verified by Smale \cite{sma61})
and the almost contact structure on $S^5$
given by $TS^5\cap\rmi TS^5$ is homotopically unique,
see \cite[Lemma 1]{mss61}.
The latter implies with Gray stability (cf.\ \cite{gei08})
and the $h$-principle (see \cite{bem15})
that up to isotopy $S^5$ admits a unique overtwisted contact structure. 
Kwon--van Koert demonstrated
that the restriction of the mean Euler characteristic
(of the $S^1$-equivariant symplectic homology)
minus $\frac12$
to a certain submonoid of contact structures on $S^5$
yields an epimorphism onto $(\Q,+)$,
see \cite[Theorem 5.22]{kvk16}.
Here,
the monoid operation is the contact connected sum
with the standard contact sphere as zero element,
cf.\ \cite{gei08}.
In this article we show
that the submonoid contains infinitely many
contact structures on $S^5$
that are not Brieskorn.
The examples are obtained by contact connected sum
of Brieskorn contact structures.

The following theorem was discovered in \cite{bck24}.

\begin{thmintr}
 \label{thmintr:monodromy}
 There are
 infinitely many pairs of
 $4$-tuples $\bfa$ and $\bfb$ of natural numbers
 such that the contact structure $\xi_{\bfa}\#\xi_{\bfb}$ on $S^5$,
 defined by contact connected sum
 of the Brieskorn contact spheres
 $\big(\Sigma(\bfa),\xi_{\bfa}\big)$ and
 $\big(\Sigma(\bfb),\xi_{\bfb}\big)$,
 is not contactomorphic to a Brieskorn contact structure.
 Furthermore the contact structures
 $\xi_{\bfa}\#\xi_{\bfb}$ are pairwise
 not contactomorphic.
\end{thmintr}

Observe, that the method used in this paper shows,
in dimensions $2n-1\geq7$,
that the subclass of Brieskorn contact structures on $S^{2n-1}$
having {\it pairwise coprime exponents}
is not closed under contact connected sum,
see Remark \ref{rem:general}.
In fact,
there are infinitely many,
pairwise non-contactomorphic such examples.


\section{Topological invariants\label{sec:topinv}}


\subsection{Brieskorn's criterion\label{subsec:brieskorncriterion}}

Let $\bfa=(a_0,\ldots,a_n)$ be an $(n+1)$-tuple
of natural numbers $a_j\geq2$ for all $j=0,\ldots,n$.
Assume that $n\geq3$.
We denote by $\Gamma(\bfa)$ the graph
with $n+1$ vertices labeled by the $a_j$, $j=0,\ldots,n$,
and an edge between $a_k$ and  $a_{\ell}$ for all $k\neq\ell$
whenever the greatest common divisor of $a_k$ and $a_{\ell}$
satisfies $\gcd(a_k,a_{\ell})\geq2$.
The connected component of $\Gamma(\bfa)$
consisting of even numbers is denoted by $\Gamma^2(\bfa)$.
In this situation \cite[Satz 1]{brk66} can be phrased as follows:
The Brieskorn manifold $\Sigma(\bfa)$ is homeomorphic to $S^{2n-1}$
if and only if one of the following conditions is satisfied:
\begin{itemize}
  \item[(i)]
    $\Gamma(\bfa)$ has at least two isolated points.
  \item[(ii)]
    $\Gamma(\bfa)$ has an isolated point and
    the connected component $\Gamma^2(\bfa)$ of $\Gamma(\bfa)$
    satisfies:
    $\#\Gamma^2(\bfa)>1$ is odd
    (hence, $\Gamma^2(\bfa)\neq\emptyset$) and
    $\gcd(a_k,a_{\ell})=2$ for all $a_k,a_{\ell}\in \Gamma^2(\bfa)$, $k\neq\ell$.
\end{itemize}
For $n=2$, the conditions for $\Gamma(\bfa)$
are equivalent to the fact that $\Gamma(\bfa)$
is an integral homology sphere.

\begin{exwith}
\label{ex:sequencewithm}
For later use we remark:
Given a natural number $m\geq2$,
a quick application of B\'ezout's diophantic identity $ax+by=1$ yields,
that the numbers $m,m+1,2m+1$ and $4m+3$ satisfy
(i) of Brieskorn's criterion with isolated points $m+1$ and $2m+1$.
In fact, they are pairwise coprime provided $\gcd(m,3)=1$.
Therefore,
as in dimension $5$ no exotic spheres exist (see \cite{sma61,km63}),
\[
\Sigma_m:=
\Sigma\big(m,m+1,2m+1,4m+3\big)
\]
is diffeomorphic to $S^5$.
\end{exwith}

\begin{exwith}
\label{ex:fermatm}
For integers $\ell=0,1,2,\ldots$,
{\bf Fermat numbers}
\[
F_{\ell}:=2^{2^{\ell}}+1
\]
can be computed recursively via
\[
F_{\ell}=F_0\cdot F_1\cdot\ldots\cdot F_{\ell-1}+2
\,,
\quad\ell\geq1
\,,
\]
where $F_{\ell}-2=F_{\ell-1}(F_{\ell-1}-2)$
is used to proceed with the inductive step.
In particular,
for $k<\ell$ we find $\gcd(F_k,F_{\ell})=1$.
Therefore,
by Brieskorn's criterion
\[
\Sigma_{\ell}:=
\Sigma\big(F_{\ell},\ldots,F_{\ell+n}\big)
\]
is a topological $(n-1)$-sphere.
\end{exwith}


\subsection{Positivity of equivariant Euler characteristic\label{subsec:posofequiveuchar}}

All Brieskorn manifolds $\Sigma(\bfa)$, $\bfa=(a_0,\ldots,a_n)$,
carry a natural fixed-point free
(finite isotropy groups are allowed)
circle action
\[
\zeta\cdot(z_0,\ldots,z_n)=
\big(\zeta^{d/a_0}z_0,\ldots,\zeta^{d/a_n}z_n\big)
\,,
\]
where $\zeta\in\C$ with $|\zeta|=1$ and
\[
d=d_{\bfa}:=\lcm(a_0,\ldots,a_n)
\,.
\]
Randell \cite{ran75} computed the corresponding
rational equivariant homology using the rational Gysin sequence,
cf.\ \cite[Section 4.2]{fschvk12}.
The related Euler characteristic equals
\[
\chi^{S^1}\big(\Sigma(\bfa)\big)=
n+(-1)^{n-1}\kappa(\bfa)
\,,
\]
see \cite[Corollary 1]{ran75},
where $\kappa(\bfa)$ denotes the rank
of the singular homology of $\Sigma(\bfa)$ in degree $n-1$
and equals
\[
\kappa(\bfa)=
\sum_{k=0}^{n+1}
\sum_{I_k}(-1)^{n+1-k}
\frac{\Pi_{j\in I_k}a_j}{\lcm_{j\in I_k}a_j}
\,,
\]
see \cite[Section 3]{ran75} or \cite[Section 3]{kvk16}.
The inner sum ranges over all subsets $I_k$ of $\{0,\ldots,n\}$
with $k\in\{0,\ldots,n+1\}$ elements.
Observe that for $k=2$ the summands are given
by $(-1)^{n+1}\gcd_{j\in I_2}a_j$,
so that the expression starts with
\[
\kappa(\bfa)=
(-1)^{n+1}
\Big(
1-(n+1)+\sum_{I_2}\gcd_{j\in I_2}a_j\mp\cdots
\Big)
\,.
\]

\begin{rem}
\label{rem:easypositive}
$\chi^{S^1}\big(\Sigma(\bfa)\big)$ is positive
whenever $n$ is odd or $\kappa(\bfa)=0$,
the latter being the case precisely when $\Sigma(\bfa)$
is a rational homology sphere
(as $\Sigma(\bfa)$ is $(n-2)$-connected
cf.\ \cite[Proposition 3.5]{kvk16}).
\end{rem}


\subsection{Invariant Brieskorn submanifolds\label{subsec:invbriesub}}

Given a Brieskorn manifold $\Sigma(\bfa)$
together with the natural circle action
from Section \ref{subsec:posofequiveuchar}
one obtains invariant submanifolds $\Sigma(\bfb)$
that themselves are Brieskorn
by freezing some but at most $(n-1)$-many coordinates $z_j$
of $\C^{n+1}$ at zero.
Here,
the tuple $\bfb$ of integers emerges from $\bfa$
by deleting all entries that correspond to a frozen coordinate.
We will indicate this writing $\bfb\subset\bfa$.

\begin{prop}
\label{prop:fullypositive}
Let $\Sigma(\bfa)$ be diffeomorphic to $S^5$,
i.e.\ $\bfa$ fulfills Brieskorn's criterion.
Then the equivariant Euler characteristic satisfies
\[
\chi^{S^1}\big(\Sigma(\bfb)\big)>0
\]
for all invariant Brieskorn submanifolds
$\Sigma(\bfb)$, $\bfb\subset\bfa$.
\end{prop}

\begin{proof}
By Remark \ref{rem:easypositive}
it is enough to show that all appearing
$3$-dimensional Brieskorn submanifolds
$\Sigma(\bfb)$, $\bfb\subset\bfa$,
are rational homology spheres.
We will show $\kappa(\bfb)=0$
for the rank of the first singular homology
\[
\kappa(\bfb)=
2-
\Big(\!\gcd(b_0,b_1)+\gcd(b_0,b_2)+\gcd(b_1,b_2)\Big)+
\frac{b_0b_1b_2}{\lcm(b_0,b_1,b_2)}
\]
of $\Sigma(\bfb)$.
Observe that $\bfb$ equals
\[
(mp,mq,r)
\quad\text{or}\quad
(2p,2q,2r)
\]
with natural numbers $p,q,r$ and $m$
such that $p,q,r$ and $m,r$, resp., are pairwise coprime.
Indeed,
deleting the isolated point required in (ii) of Brieskorn's criterion
(see Section \ref{subsec:brieskorncriterion})
yields $\bfb=(2p,2q,2r)$;
otherwise, we obtain $\bfb=(2p,2q,r)$.
Starting with $\bfa$ fulfilling (i) one gets $\bfb=(mp,mq,r)$
by deleting one of the isolated points.
Here,
$m$ is the greatest common divisor of the two remaining,
potentially related $a_j$'s.
Keeping two isolated points yields $\bfb=(p,q,r)$.
Hence,
\[
\kappa(\bfb)=
2-(m+2)+\frac{m^2pqr}{mpqr}=0
\quad\text{and}\quad
\kappa(\bfb)=
2-3\!\cdot\!2+\frac{2^3pqr}{2pqr}=0
\,,
\]
resp., as claimed.
\end{proof}

\begin{rem}
\label{rem:generalpositive}
Positivity of the equivariant Euler characteristic
$\chi^{S^1}\big(\Sigma(\bfb)\big)$
for all odd-dimensional invariant Brieskorn submanifolds
$\Sigma(\bfb)$, $\bfb\subset\bfa$,
of a general Brieskorn sphere $\Sigma(\bfa)$
can be ensured by requiring all $\Sigma(\bfb)$
to be rational homology spheres.
By Brieskorn's criterion this will be satisfied
provided that the $a_j$'s in $\bfa$ are pairwise coprime.
\end{rem}


\section{Contact homological invariant\label{sec:conthomoinv}}

The Brieskorn contact structure $\xi_{\bfa}$ on $\Sigma(\bfa)$
up to contact isotopy can be given by the kernel of the
contact form defined by
\[
\frac{\rmi}{4\pi}
\sum_{j=0}^n
a_j
\big(
	z_j\rmd\bar{z}_j-\bar{z}_j\rmd z_j
\big)\,
\Big|_{T\Sigma(\bfa)}
\,,
\]
see \cite[Section 7.1]{gei08}.
The corresponding Reeb vector field
\[
2\pi\rmi
\sum_{j=0}^n
\frac{1}{a_j}
\big(
	z_j\partial_{z_j}-\bar{z}_j\partial_{\bar{z}_j}
\big)\,
\Big|_{\Sigma(\bfa)}
\]
induces the flow
\[
\Phi_t(z_0,\ldots,z_n)=
\Big(\rme^{2\pi\rmi t/a_0}z_0,\ldots,\rme^{2\pi\rmi t/a_n}z_n\Big)
\]
on $\Sigma(\bfa)$.
Setting $\zeta=\rme^{2\pi\rmi t/d_{\bfa}}$
we obtain the circle action considered in Section \ref{subsec:posofequiveuchar},
so that the Reeb flow $\Phi_t$ is $d_{\bfa}$-periodic.


\subsection{Positivity of the mean Euler characteristic\label{subsec:posmeaneuchar}}

Denote by
\[
\Sigma_T:=\big\{\bfz\in\Sigma(\bfa)\:\big|\:\Phi_T(\bfz)=\bfz\big\}
\]
the submanifold
of $\Sigma(\bfa)$ that consists of all $T$-periodic Reeb orbits.
Observe, that the period $T$ necessarily will be an integer.
We assume $T$ to be positive.
Furthermore
the orbit space $\Sigma_T$ is equal to
the invariant Brieskorn submanifold
$\Sigma(\bfb)$, $\bfb\subset\bfa$,
obtained by setting $z_j=0$ for all $j\in\{0,\ldots,n\}$
for which $T/a_j$ is not an integer.
The dimension of $\Sigma_T$ equals
\[
\dim(\Sigma_T)=2m_T-3
\,,
\]
where $m_T$ is the number of $a_j$ that divide $T$.
By \cite[formula (4) on p.~250]{ueb16}
the Robbin--Salamon index of $\Sigma_T$ equals
\[
\mu_{\RS}(\Sigma_T)=\sum_{j = 0}^n
\left(
	\left\lfloor\frac{T}{a_j}\right\rfloor+
	\left\lceil\frac{T}{a_j}\right\rceil
\right)-
2T\,.
\]

\begin{lem}
\label{lem:determinthesign}
\[
\mu_{\RS}(\Sigma_T)\equiv n+1-m_T\:(\mathrm{mod}\:2)
\,.
\]
\end{lem}

\begin{proof}
If $T/a_j$ is an integer,
then
\[
\left\lfloor\frac{T}{a_j}\right\rfloor+
\left\lceil\frac{T}{a_j}\right\rceil=
2\cdot\frac{T}{a_j}
\,.
\]
Otherwise,
\[
\left\lfloor\frac{T}{a_j}\right\rfloor+
\left\lceil\frac{T}{a_j}\right\rceil
\equiv
1\:(\mathrm{mod}\:2)
\,.
\]
Therefore, each $a_j$ not dividing $T$,
contributes $1\:(\mathrm{mod}\:2)$ to $\mu_{\RS}(\Sigma_T)$.
\end{proof}

We denote by $\Sigma_T/S^1$
the space of $T$-periodic Reeb orbits
up to parametrisation,
i.e.\ the quotient orbifold by the natural action,
which has dimension $2m_T-4$.

\begin{cor}
\label{cor:determinthesign}
\[
\mu_{\RS}(\Sigma_T)-
\frac12\dim(\Sigma_T/S^1)
\equiv n+1\:(\mathrm{mod}\:2)
\,.
\]
\end{cor}

The minimal periods $T_1<\ldots<T_k=d_{\bfa}$
of the Reeb flow $\Phi_t$ are natural numbers $T_j$
that divide $d_{\bfa}$.
Following \cite{fschvk12} and \cite[Section 5.8]{kvk16}
define the frequencies
$\phi_{T_i}(T_{i+1},\ldots,T_k)$
to be the number of natural numbers $a$
such that $aT_i <T_k$ and $aT_i\notin T_j\N$
for all $j\in\{i+1,\ldots,k\}$.
By convention we have $\phi_{d_{\bfa}}(\emptyset)=1$.

The mean Euler characteristic $\chim(\xi_{\bfa})$
induced by the $S^1$-equivariant symplectic homology
defined in \cite{fschvk12} was computed
in \cite[Proposition 5.20]{kvk16} for the present context:
Whenever $\mu_{\RS}\big(\Sigma(\bfa)\big)\neq0$
the mean Euler characteristic $\chim(\xi_{\bfa})$
is defined, an invariant of the isomorphy class
of the contact structure $\xi_{\bfa}$,
and equals
\[
\chim(\xi_{\bfa})=
 \frac{
 	\sum_{i = 1}^k
	(-1)^{
		\mu_{\RS}(\Sigma_{T_i})-
		\frac12\dim(\Sigma_{T_i}/S^1)}
	\phi_{T_i}(T_{i+1},\ldots,T_k)\cdot
	\chi^{S^1}(\Sigma_{T_i})
	}
 	{
		 |\mu_{\RS}\big(\Sigma(\bfa)\big)|
	}
\,.
\]
Corollary \ref{cor:determinthesign} implies:
\[
\chim(\xi_{\bfa})=
(-1)^{n+1}
 \frac{
 	\sum_{i = 1}^k
	\phi_{T_i}(T_{i+1},\ldots,T_k)\cdot
	\chi^{S^1}(\Sigma_{T_i})
	}
 	{
		 |\mu_{\RS}\big(\Sigma(\bfa)\big)|
	}
\,.
\]
Notice, that \cite[formula (31) on p.~265]{ueb16} gives
\[
\mu_{\RS}\big(\Sigma(\bfa)\big)
=2d_{\bfa}\cdot
\left(
	\sum_{j = 0}^n\frac{1}{a_j}-1
\right)
\,.
\]

\begin{lem}
\label{lem:meaneuchrwelldef}
The mean Euler characteristic $\chim(\xi_{\bfa})$
is defined for all Brieskorn manifolds $\Sigma(\bfa)$,
for which $\bfa$ admits an isolated exponent.
\end{lem}

\begin{proof}
Arguing by contradiction we assume that 
$\mu_{\RS}\big(\Sigma(\bfa)\big)=0$,
which is equivalent to
$\frac{1}{a_0}+\ldots+\frac{1}{a_n}=1$.
This implies $a_j\geq2$ for all $j=0,\ldots,n$.
By assumption,
one of the exponents, say $a_0$,
satisfies $\gcd(a_0,a_j)=1$ for all $j=1,\ldots,n$.
Multiplication with $d_{\bfa}$ yields:
$d_{\bfa}/a_0=d_{\bfa}-d_{\bfa}/a_1-\ldots-d_{\bfa}/a_n$.
Therefore, $a_0^2$ divides $d_{\bfa}$,
which contradicts $d_{\bfa}$ being
the least common multiple of the $a_j$
for all $j=0,\ldots,n$.
\end{proof}

\begin{prop}
\label{prop:contactpositive}
Let $\Sigma(\bfa)$ be a $5$-dimensional Brieskorn sphere.
Then
\[
\chim(\xi_{\bfa})> 0
\,.
\]
\end{prop}

\begin{proof}
Follows with the positivity statement in Proposition \ref{prop:fullypositive}.
\end{proof}

\begin{rem}
\label{rem:generalcontactpositive}
In view of Remark \ref{rem:generalpositive}
or the formula at the beginning of Section \ref{subsec:pairwisecoprim},
we obtain positivity of
$(-1)^{n+1}\chim(\xi_{\bfa})$ in all dimensions,
whenever the exponents of $\Sigma(\bfa)$ are pairwise coprime.
\end{rem}


\subsection{Pairwise coprime\label{subsec:pairwisecoprim}}

By \cite[Proposition 4.6]{fschvk12} or \cite[Proposition 5.21]{kvk16}
the mean Euler characteristic $\chim(\xi_{\bfa})$ of Brieskorn manifolds
$\Sigma(\bfa)$ with pairwise coprime exponents $a_0,\ldots,a_n$
equals:
\[
(-1)^{n+1}
\frac{
	n+
	(n-1)\sum_{j_0}(a_{j_0}-1)+
	\cdots+
	\sum_{j_0<\cdots<j_{n-2}}(a_{j_0}-1)\cdots(a_{j_{n-2}}-1)
	}
	{
	2|(\sum_ja_0\cdots\hat{a}_j\cdots a_n) - a_0\cdots a_n|
	}
\,.
\]

\begin{exwith}
\label{ex:below14}
Let $\bfa=\big(m,m+1,2m+1,4m+3\big)$
for a natural number  $m$ not divisible by $3$.
Abbreviate the contact structure $\xi_{\bfa}$ on $\Sigma_m$
by $\xi_m$, see Example \ref{ex:sequencewithm}.
We obtain
\[
\chim(\xi_m)=
\frac{21m^2+17m+3}{16m^4-8m^3-50m^2-34m-6}
\,.
\]
Observe that 
\[
\chim(\xi_4)=\frac{407}{2642}<\frac14
\]
and that the function $m\mapsto\chim(\xi_m)$
decreases strictly for $m\geq4$.
Indeed,
$\chim(\xi_m)$,
viewed as a rational function over the complex numbers,
has all its poles on the open disc of radius $3$.
This follows with Rouch\'e's theorem as $16\cdot3^4=1296$
dominates the absolut value of the lower order terms of the denominator
$h(m)$ evaluated along $|m|=3$,
which is bounded by $216+450+102+6=774$.
Furthermore,
denoting the nominator by $g(m)$,
the polynomial $(g'h-h'g)(m)$ taken over the reals
equals $-672m^5-648m^4+80m^3+208m^2+48m$
and is, therefore, negative for $m\geq1$.
\end{exwith}

\begin{rem}
\label{rem:generalpmdecreasing}
Assigning to each natural number $\ell$
the $(n+1)$-tuple $\bfa_{\ell}$
of consecutive Fermat numbers
\[
a_j=F_{\ell+j}=x^{2^j}+1
\,,\quad\text{with}\,\,
x=2^{2^{\ell}}
\,,
\]
as in Example \ref{ex:fermatm},
we see that $\chim(\xi_{\bfa_{\ell}})$ is a rational function
in $x$,
whose asymptotic behaviour is given by
\[
(-1)^{n+1}\frac{1}{2(a_0-1)(a_1-1)}=
(-1)^{n+1}\frac{1}{2x^3}
\,.
\]
Hence,
for sufficiently large $\ell$ we find that
$\ell\mapsto(-1)^{n+1}\chim(\xi_{\bfa_{\ell}})$
is strictly decreasing taking values in $(0,1/4)$.
\end{rem}


\subsection{Monoidomorphic contact connected sum\label{subsec:monoidomorphic}}

In \cite[Section 5.9]{kvk16} and \cite[Proposition 6]{ueb16}
additivity properties of $\chim+(-1)^n\frac12$
under contact connected sum were discussed:
Let $\Sigma({\bfa})$ and $\Sigma({\bfb})$ be Brieskorn manifolds
of dimension $2n-1$,
for which the mean Euler characteristic is defined.
Then the mean Euler characteristic
for the contact connected sum of
$\big(\Sigma(\bfa),\xi_{\bfa}\big)$ and
$\big(\Sigma(\bfb),\xi_{\bfb}\big)$
is defined and equals
\[
\chim(\xi_{\bfa}\#\xi_{\bfb})=
\chim(\xi_{\bfa})+
\chim(\xi_{\bfb})+
(-1)^n\frac12
\,.
\]

\begin{proof}[Proof of Theorem \ref{thmintr:monodromy}]
For $m\geq4$ the contact invariant $\chim(\xi_m)>0$
is strictly decreasing in $m$,
see Example \ref{ex:below14}.
Hence,
the Brieskorn contact structures $\xi_m$
are pairwise non-contactomorphic.
Moreover,
$\chim(\xi_4)<\frac14$ implies
$\chim(\xi_m\#\xi_m)<0$
for all $m\geq4$ by monoidomorphicity.
Therefore, with Proposition \ref{prop:contactpositive},
the contact structures $\xi_m\#\xi_m$ on $S^5$
are not contactomorphic to a Brieskorn contact structure $\xi_{\bfa}$
for all $m\geq4$.
Because of
$\chim(\xi_m\#\xi_m)=2\chim(\xi_m)-1/2$
the non-Brieskorn contact structures $\xi_m\#\xi_m$ themselves
are pairwise non-contactomorphic for different integers $m\geq4$.
\end{proof}

\begin{rem}
\label{rem:general}
In view of Remark \ref{rem:generalpmdecreasing}
we obtain for sufficiently large natural numbers $\ell$
that $(-1)^{n+1}\chim(\xi_{\bfa_{\ell}})<1/4$.
With monoidomorphicity
\[
\chim(\xi_{\bfa_{\ell}}\#\xi_{\bfa_{\ell}})=
2\chim(\xi_{\bfa_{\ell}})+
(-1)^n\frac12
\]
this implies
\[
(-1)^{n+1}\chim(\xi_{\bfa_{\ell}}\#\xi_{\bfa_{\ell}})<0
\,.
\]
This contrasts the fact that
$(-1)^{n+1}\chim(\xi_{\bfa_{\ell}})>0$
by Remark \ref{rem:generalcontactpositive}.
\end{rem}


\section{Inverse elements\label{sec:invelem}}

Contact structures on $S^{2n-1}$
form an abelian monoid under contact connected sum.
The standard contact structure
$TS^{2n-1}\cap\rmi TS^{2n-1}$
appears as the neutral element.
A natural question is about invertibility
of contact structures within the monoidal structure.
For $n=2$ no non--trivial element
will be invertible by overtwistedness.
For $n\geq3$,
the Eliashberg--Floer--McDuff theorem implies
that no potential inverse of a Brieskorn contact structure
is a Brieskorn contact structure itself.
Indeed,
boundary connected sum with the corresponding
natural critical Stein fillings would yield
a symplectically aspherical filling
of $TS^{2n-1}\cap\rmi TS^{2n-1}$,
which in turn will be a smooth ball,
see \cite{mcd91,bgz19,zhou23}.
But this would contradict criticality
of the boundary connected sum of 
the natural critical Stein fillings
in view of the Mayer--Vietoris sequence.


\begin{ack}
We thank Peter Albers, Kai Cieliebak, Urs Frauenfelder,
Hansj\"org Geiges, Myeonggi Kwon, Agustin Moreno
and the participants of the A5/C5-Seminar
Jan Eyll and Lars Kelling. 
\end{ack}


\end{document}